\documentclass[leqno,11pt]{article}

\usepackage{amsmath,amsthm,verbatim}
\usepackage[psamsfonts]{amssymb}
\usepackage[final]{graphicx}
\usepackage{multicol}
\usepackage{latexsym}
\usepackage{enumerate}
\usepackage{times}
\usepackage[T1]{fontenc}

\newtheorem{teor}{Theorem}[section]
\newtheorem{prop}[teor]{Proposition}

\theoremstyle{remark}

\theoremstyle{definition}

\newtheorem{defi}{Definition}[section]

\DeclareMathOperator{\dom}{dom}

\title{\textbf{An Approach to Differential Invariants of $G\,$-\,Structures}
\footnote{
\newline
\indent \indent
Supported by Junta de Andaluc\'{\i}a (Spain),
under grant P.A.I. ID Code FQM-324.}
}

\author{\textsc{Ignacio S\'anchez-Rodr\'{\i}guez}
}

\date{}

\begin{document}
\maketitle

\begin{abstract}
\noindent 
In a natural way, the diffeomorphisms of a manifold onto itself act on the
reference frame bundles of any order and on the bundles associated with them.
Due to the transitivity, the invariants by diffeomorphisms of an associated 
bundle correspond to the real functions on the quotient space of the typical fiber 
by the action that defines the associated bundle.

Scalar differential invariants are functions on bundles of jets of sections of 
associated bundles, which are invariant by the standard action of diffeomorphisms. 
We show that these jet bundles are, in turn, associated bundles with reference
frame bundles of higher order. By the previous result, this will allow to recognize 
scalar differential invariants as real functions over a quotient space, 
independently of the base manifold.

We apply the above to $G$-structures, and describe their scalar differential invariants
as functions over a certain quotient space. 
This quotient space contains a dense differentiable manifold, whose dimension is 
``the number of functionally independent scalar differential invariants''. 
We will obtain a lower bound of that number.
\end{abstract}

\section{Introduction}
\label{intro}

The paradigm of scalar differential invariant of a
$G$-structure is the scalar curvature
$\textbf{S}_{\mathbf{g}}\colon M\to \mathbb{R}$ of a 
Riemannian manifold $(M,\mathbf{g})$. 
On each $m\in M$ it is defined
$\mathbf{S}_{\mathbf{g}}(m)$ from the values of $\mathbf{g}$
and its partial derivatives in $m$ up to second order,
that is to say, it depends of the 2-jet of the metric at $m$.

It is said that \emph{the scalar curvature
	is invariant by diffeomorphisms} because
if $\varphi \colon M\to M$ is a 
diffeomorphism then the scalar curvature
of the metric $\varphi ^{-1 *}\mathbf{g}$ verifies 
$\mathbf{S_{\varphi ^{-1 *}\mathbf{g}}}=
	\mathbf{S}_{\mathbf{g}}\circ\varphi ^{-1}$.

The scalar differential invariants of the
	metrics are well studied and it is known
the number of functionally independent metric invariants
-- 	depending on the dimension of $M$ and the
maximal order of derivatives of $\mathbf{g}$
involved -- and how they are obtained from the
Riemann curvature.

This essay is a preparatory study in order to access the
scalar differential invariants of other $G$-structures types.
Particularly, we are interested in the
differential invariants of
	conformal structures -- do not confuse these with
the \emph{conformal invariants with weight}; in any case,
the invariants of our work are of \emph{weight zero}.

\section{Notations and basic concepts.}
 \begin{itemize}
	\item $P$, principal bundle over $M$ with 
	group $G$; right action $p\cdot g$;  $P/G=M$.
	\item
	$W$, manifold; left action $g\cdot w$; 
	$G\backslash W$ -- topological space, 
	in general not a manifold;
	notation: $[w]$.
	\item
	$P\times W$; right action $(p,w)\cdot 
		g :=(p\cdot g,g^{-1}\cdot w)$; we obtain the associated bundle
	$P(W)\equiv (P\times W)/G$; notation: $[p,w]$.
	\item $LM$, linear frame bundle
	of $M$ with group $\mathrm{Gl}_n$.
	\item
	$\mathrm{Gl}_n/\mathrm{O}_n$, quotient manifold; 
	left action of $\mathrm{Gl}_n$ on 
	$\mathrm{Gl}_n/\mathrm{O}_n$.
	\item
	$LM(\mathrm{Gl}_n/\mathrm{O}_n)$,
	\emph{the bundle of metrics}: 
	Giving a metric $\mathbf{g}$ is equivalent
	to give a section $\boldsymbol{\sigma}_{\mathbf{g}} 
		\colon M\to LM(\mathrm{Gl}_n/\mathrm{O}_n)$, 
	$\boldsymbol{\sigma}_{\mathbf{g}} (m)=[l,\mathrm{O}_n]$,
	with $l\in LM$ being some basis 
	$\mathbf{g}$-orthonormal.
	Reciprocally, a section $\boldsymbol{\sigma}$ defines a
	metric establishing which are the orthonormal bases.
	\item The bundle 
	$J^rLM(\mathrm{Gl}_n/\mathrm{O}_n)$ of $r$-jets of
	sections of $LM(\mathrm{Gl}_n/\mathrm{O}_n)$ has, 
	for each element
	$j^r_m\boldsymbol{\sigma} $, the information 
	of the $r$-jet at $m$ of the corresponding metric 
	$\mathbf{g}$.
	\item We can define the \emph{
		scalar curvature function}
	$\textbf{S}\colon J^2LM(\mathrm{Gl}_n/\mathrm{O}_n)\to 
		\mathbb{R}$ so that $\textbf{S}\circ j^2
		\boldsymbol{\sigma}_{\mathbf{g}} =\textbf{S}_{\mathbf{g}}$ 
	-- that is $\textbf{S}( j^2_m
		\boldsymbol{\sigma}_{\mathbf{g}}) :=\textbf{S}_{\mathbf{g}}(m)$.
	\item Action of a diffeomorphism $\varphi \colon M\to M$
	\begin{itemize}
		\item[] on $LM$: 
		$\qquad\qquad\qquad\qquad\qquad\qquad\quad 
			l\mapsto \varphi_{*}\circ l$.
		\item[] on $LM(W)$: 
		$\qquad\qquad\qquad\qquad\Bar{\varphi }\colon
			[l,w]\mapsto [\varphi_{*}\circ l,w]$.
		\item[] on sections $\boldsymbol{\sigma}\colon M\to LM(W)$:
		$\qquad\boldsymbol{\sigma}\mapsto \Bar{\varphi }
			\circ\boldsymbol{\sigma}\circ\varphi^{-1}$.
		\item[] on jets of sections: $\qquad \qquad
			\widehat{\varphi }^{r}\colon j^r_m
			\boldsymbol{\sigma}\mapsto j^r_{\varphi(m)}(\Bar{\varphi }
			\circ\boldsymbol{\sigma}\circ\varphi^{-1})$.
	\end{itemize}
	\item In the case $W=\mathrm{Gl}_n/\mathrm{O}_n$,
	it is easy to prove that 
	$\boldsymbol{\sigma}_{\varphi ^{-1 *}\mathbf{g}}=
		\Bar{\varphi }\circ\boldsymbol{\sigma}_{\mathbf{g}}\circ\varphi^{-1}$ 
	and, using this fact, to prove that 
	$$\mathbf{S_{\varphi ^{-1 *}\mathbf{g}}}=
		\mathbf{S}_{\mathbf{g}}\circ\varphi ^{-1},
		\ \forall\,\mathbf{g},\ \forall\,\varphi\quad
		\Longleftrightarrow \quad\textbf{S}=\textbf{S}\circ
		\widehat{\varphi }^{2},\ \forall\,\varphi .$$
\end{itemize}

\section{Scalar differential invariants}

\subsection{The action of diffeomorphisms on frame bundles and associated bundles}
 Let $F^rM$ be the \emph{ $r$-th order
	frame bundle over} $M$, with $\dim M=n$. This is a principal bundle
with group $\mathrm{G}^r_n$, 
the \emph{$r$-th jet group}.
The right action of
$j^r_{_0}\xi \in \mathrm{G}^r_n$ on
$j^r_{_0}\psi \in F^rM$ is defined by
$j^r_{_0}(\psi \circ \xi ) \in F^rM$,
with $\xi $ being a diffeomorphism between neighborhoods of
	$0\in \mathbb{R}^n$, with $\xi (0)=0$
and $\psi $ being a diffeomorphism between a neighborhood
	of $0$ and an open set of $M$.

The pseudogroup $\mathcal{D}M$ of
	diffeomorphisms between open sets of $M$ acts on $F^rM$:
The left action of
$\varphi \in \mathcal{D}M$ on
$j^r_{_0}\psi \in F^rM$ is defined by
$j^r_{_0}(\varphi \circ \psi )\in F^rM$,
when $\psi (0)\in \dom\varphi $.

This action of diffeomorphisms is transitive:
$$\forall\, j^r_{_0}\psi, j^r_{_0}\psi^{\prime}\in F^rM,\
\exists \,\varphi \in \mathcal{D}M\ \text{such that}\
j^r_{_0}(\varphi \circ \psi )=j^r_{_0}\psi^{\prime}\,.$$

Let $F^rM(W)$ be
an associated bundle to $F^rM$; 
the left action of 
$\varphi \in \mathcal{D}M$ defined there 
is given, over $\dom\varphi $, by:
$$\Bar{\varphi }^r
	\colon F^rM(W)\to F^rM(W),\quad
	[j^r_{_0}\psi ,w]\mapsto
	[j^r_{_0}(\varphi \circ \psi ),w].$$

\subsection{Invariants of associated bundles}

\begin{defi}[Invariants of associated bundles]
A function
$f\colon F^rM(W)\to \mathbb{R}$
is a \textit{scalar differential invariant}
	of $F^rM(W)$ if,
$\forall\, w\in W$, $\forall\,\varphi \in \mathcal{D}M$ and
$\forall\, j^r_{_0}\psi \in F^rM$ with
$\psi 0\in \operatorname{dom}\varphi $, it is verified:
$$f[j^r_{_0}\psi ,w]=f[j^r_{_0}(\varphi \circ \psi ),w].$$
Equivalently, $\forall\,\varphi \in \mathcal{D}M$,
$f\circ \Bar{\varphi }^r=f$ over
$\operatorname{dom}\varphi $.
\end{defi}
We can reduce the problem of finding
scalar differential invariants of $F^rM(W)$
to a problem that is \emph{independent from} $M$ and its
diffeomorphisms.
\begin{teor}[Independence from $M$ of the invariants]
The set of scalar differential invariants
of $F^rM(W)$ is in bijective correspondence with the set of
functions $h\colon \mathrm{G}^r_n\backslash W\to \mathbb{R}$
such that $h\circ \Pi $ is differentiable -- being
$\Pi \colon W\to \mathrm{G}^r_n\backslash W$,
$\ w\mapsto [w]$.
\end{teor}
\begin{proof}
	Given an invariant $f\colon F^rM(W)\to \mathbb{R}$, 
	the function
	$h[w]:=f[j^r_{_0}\psi ,w]$
	is well defined because:
	
	(\textit{i}) for another
	$j^r_{_0}\psi ^\prime $,
	we obtain a diffeomorphism
	$\varphi =\psi ^\prime \circ\psi ^{-1}$
	between neighborhoods of
	$\psi 0$ and $\psi ^\prime 0$, and then
	$f[j^r_{_0}\psi ^\prime ,w]
	=f[j^r_{_0}(\varphi \circ \psi ),w]
	=f[j^r_{_0}\psi ,w]$;
	
	(\textit{ii}) if $w^\prime \in [w]$,
	there exists $j^r_{_0}\xi \in \mathrm{G}^r_n$
	such that $[j^r_{_0}\psi ,w^\prime ]=
	[j^r_{_0}\psi ,j^r_{_0}\xi \cdot w]=
	[j^r_{_0}(\psi \circ \xi ),w ]$, then
	$f[j^r_{_0}\psi ,w^\prime ]=
	f[j^r_{_0}(\psi \circ \xi ),w ]$ and,
	from (\textit{i}),
	$f[j^r_{_0}\psi ,w^\prime ]=
	f[j^r_{_0}\psi ,w ]$.
	
	The differentiability follows from the identity
	$h\circ \Pi =f\circ \pi \circ \iota _z$,
	where $\pi \colon F^rM\times W\to F^rM(W)$
	is the natural projection and
	$\iota _z\colon W\to F^rM\times W$,
	$\iota _z(w):=(z,w)$, with $z\in F^rM$.
	
	Reciprocally, given
	$h\colon \mathrm{G}^r_n\backslash W\to \mathbb{R}$,
	the function $f\colon F^rM(W)\to \mathbb{R}$ defined by
	$f[j^r_{_0}\psi ,w ]:=h[w]$,
	$\forall\, j^r_{_0}\psi \in F^rM$, 
	which obviously is invariant,
	is well defined because
	$f[j^r_{_0}(\psi \circ \xi),j^r_{_0}\xi ^{-1}\cdot w ]=
	h[j^r_{_0}\xi ^{-1}\cdot w]=h[w],\quad
	\forall \,j^r_{_0}\xi \in \mathrm{G}^r_n$
	
	It is known that $f$ is differentiable if and only if
	$f\circ \pi $ is differentiable; as
	$f\circ \pi =h\circ \Pi \circ \pi _{_2}$,
	with $\pi _{_2}\colon F^rM\times W\to W$
	being the projection $\pi _{_2}(z,w)=w$,
	the differentiability of $f$ follows of
	the differentiability of $h\circ \Pi$.
\end{proof}

\section{Jets of sections of associated bundles}
\subsection{Bundles of $r$-jets of sections}
Given a bundle $E$ over 
$M$ with fiber $W$,
it is obtained the \emph{bundle of $r$-jets of
	local sections} of $E$,
$J^rE$, whose basis is $M$ and fiber is 
the space $J^r_{_0}(\mathbb{R}^n,W)$
of $r$-jets at $0$ of
applications of $\mathbb{R}^n$ to $W$.

Let us see how work in the case 
$E=F^kM(W)$. We use the natural section
of $F^kM$ induced from a chart $(x,U)$:
\begin{equation*}
	\widehat{x}^{k}\colon U \to F^kM,
	\quad
	p\mapsto \widehat{x}^{k}p:=
	j^k_{_0}(x^{-1}\circ \tau _{x(p)}),
	\end{equation*}
with
$\tau _{x(p)}$ being the translation
by $x(p)$ on $\mathbb{R}^n$.
A local trivialization of $E$ is:
$$\Psi ^x\colon E|_U\to U\times W,\quad
	[\widehat{x}^{k}p,w]\mapsto (p,w).$$

Now, a local section
$\sigma  \colon U\to E$ is
characterized by the application
$\sigma ^x\colon U\to W$
such that $\Psi ^x(\sigma (p))=(p,\sigma ^x(p))$.
We can write
$\sigma =[\widehat{x}^{k},\sigma ^x]$.
Then, a local trivialization of $J^rE$ is
$$J^rE|_U\to U\times J^r_{_0}(\mathbb{R}^n,W),\quad
	j^r_p\sigma \mapsto \bigl(p,j^r_{_0}(\sigma ^x\circ
	x^{-1}\circ \tau _{xp})\bigr).$$

\subsection{The bundle of jets of sections is an associated bundle}
\begin{prop}[1]
	If $E=F^kM(W)$ is an associated bundle to $F^kM$
	with typical fiber $W$
	then $J^rE$ is an associated bundle to $F^{r+k}M$
	with typical fiber $J^r_{_0}(\mathbb{R}^n,W)$.
\end{prop}

\begin{proof}
We consider the following action of
$j^{r+k}_{_0}\xi \in \mathrm{G}^{r+k}_n$ on
$j^r_{_0}\mu \in J^r_{_0}(\mathbb{R}^n,W)$ defined by
\begin{equation}
	\label{f2}
	j^r_{_0}\bigl( (j^k\xi \cdot \mu )
	\circ \xi ^{-1}\bigr) \in J^r_{_0}(\mathbb{R}^n,W),
	\end{equation}
in which occurs the $W$-valued function given by
\begin{equation*}
(j^k\xi \cdot \mu )(v):=j^k_{_0}
(\tau _{-\xi (v)}\circ \xi \circ \tau _v)\cdot \mu (v),
\quad v\in\mathbb{R}^n;
\end{equation*}
the last dot here refers
to the action of
$\mathrm{G}^k_n$ on $W$.

It is proved that (1)
defines a left action of
$\mathrm{G}^{r+k}_n$ on
$J^r_{_0}(\mathbb{R}^n,W)$ and then it produces an associated bundle
$F^{r+k}M(J^r_{_0}(\mathbb{R}^n,W))$.
We can define a fibered application
-- in the domain of the chart $x$ -- by
\begin{eqnarray*}
		\label{f3}
		\Lambda\colon J^rE &\longrightarrow & 
		F^{r+k}M(J^r_{_0}(\mathbb{R}^n,W))\\
		j^r_p\sigma &\longmapsto &\bigl[\widehat{x}^{r+k}p,\,
		j^r_{_0}(\sigma ^x\circ
		x^{-1}\circ \tau _{x(p)})\bigr],
\end{eqnarray*}
which is an isomorphism of bundles. We proved that this definition 
is independent of the chart and, thus, that it is a bundle isomorphism globally defined.
\end{proof}

\subsection{Action of diffeomorphisms on jets of sections}
In a bundle $E=F^kM(W)$ it is defined the action of a diffeomorphism $\varphi $ by:
$$\Bar{\varphi }^k\colon E\to E,\
	[j^k_{_0}\psi ,w]\mapsto
	[j^k_{_0}(\varphi \circ \psi ),w].$$
If $\sigma $
is a section of $E$ in a neighborhood of $p\in M$ then
$\Bar{\varphi }^k\circ \sigma
\circ \varphi ^{-1}$ is a section
of $E$ in a neighborhood of $\varphi (p)$. Thus, it is defined the action of $\varphi $
on $J^rE$ by:
\begin{equation*}
	\widehat{\varphi }^{r,k}\colon
	J^rE \to J^rE,
	\quad
	j^r_p\sigma \mapsto
	j^r_{\varphi (p)}(\Bar{\varphi }^k
	\circ \sigma \circ \varphi ^{-1}).
	\end{equation*}
\vspace{-3mm}
\begin{defi}[Invariants of bundles of jets of sections]
	A differentiable function
	$f\colon J^r\big(F^kM(W)\big) \to \mathbb{R}$
	is a \textit{scalar differential invariant (SDI)
			of $r$-th order over} $F^kM(W)$ if,
	$\forall \varphi \in \mathcal{D}M$,
	$f\circ \widehat{\varphi }^{r,k}=f$
	on $\dom \varphi $.
\end{defi}

The two definitions of SDI
of $F^{r+k}M(J^r_{_0}(\mathbb{R}^n,W))$
and of SDI
of $r$-th order over $F^kM(W)$
are equivalent because of the
commutative diagram:
\begin{equation*}
	\label{f8}
	\begin{array}{ccc}
	J^r\big(F^kM(W)\big) &
	\xrightarrow{\Lambda } &
	F^{r+k}M(J^r_{_0}(\mathbb{R}^n,W)) \\
	\Big\downarrow\vcenter{%
		\rlap{$\scriptstyle{\widehat{\varphi }^{r,k}}$}} & &
	\Big\downarrow\vcenter{%
		\rlap{$\scriptstyle{\Bar{\varphi }^{r+k}}$}}\\
	J^r\big(F^kM(W)\big) &
	\xrightarrow{\Lambda } &
	F^{r+k}M(J^r_{_0}(\mathbb{R}^n,W)).
	\end{array}
	\end{equation*}

\section{Scalar differential invariants of $G$-structures}
\subsection{Bundle of $G$-structures}
Let $G$ be a closed subgroup of $\mathrm{Gl}_n$.
The left action of $\mathrm{Gl}_n$ on
$\mathrm{Gl}_n/G$ gives the associated bundle
$M_G\equiv LM(\mathrm{Gl}_n/G)$
which is called the
\textit{bundle of $G$-structures}.
Its sections, $\sigma \colon M\to M_G$,
are in correspondence
with the principal subbundles, 
$P=\{l\in LM\colon [l,G]\in \sigma (M)\}$,
with group $G$ which are the so-called
$G$-structures on $M$.

Let $J^rM_G$ be the bundle of $r$-jets
of local sections of $M_G$,
whose fiber is
$J^r_{_0}(\mathbb{R}^n,\mathrm{Gl}_n/G)$.
The action of $\varphi \in\mathcal{D}M$
on $J^rM_G$ is defined by:
\begin{equation*}
	\widehat{\varphi }^{r,1}\colon
	J^rM_G \to J^rM_G,
	\quad
	j^r_p\sigma \mapsto
	j^r_{\varphi (p)}(\Bar{\varphi }
	\circ \sigma \circ \varphi ^{-1}).
	\end{equation*}
\vspace{-3mm}
\begin{defi}[Invariants of $G$-structures]
	A \textit{scalar differential invariant of $r$-th order of $G$-structures
			on} $M$
	is a differentiable function
	$f\colon J^rM_G\to \mathbb{R}$
	which verifies $f\circ \widehat{\varphi }^{r,1}=f$
	over $\dom \varphi $, $\forall \varphi \in \mathcal{D}M$.
\end{defi}
[Added in translation: An example is the scalar curvature of Riemannian geometry
		$\textbf{S}\colon J^2M_{\mathrm{O}_n}\to 
			\mathbb{R}$, being 
		$M_{\mathrm{O}_n}\equiv LM(\mathrm{Gl}_n/\mathrm{O}_n)$
		(see the Introduction)]

The isomorphism $\Lambda $ of Proposition (1) allows to
identify the bundle of jets of sections $J^rM_G$ and 
the associated bundle $F^{r+1}M
	\bigl(J^r_{_0}(\mathbb{R}^n,\mathrm{Gl}_n/G)\bigr)$ with respect to the  action of
	$j^{r+1}_{_0}\xi \in \mathrm{G}^{r+1}_n$ on
	$j^r_{_0}\mu \in J^r_{_0}(\mathbb{R}^n,\mathrm{Gl}_n/G)$ defined by
\begin{equation*}
	j^r_{_0}\bigl( (D\xi \cdot \mu)
	\circ \xi ^{-1}\bigr) \in J^r_{_0}(\mathbb{R}^n,\mathrm{Gl}_n/G),
	\end{equation*}
being $(D\xi \cdot \mu)(v):=D\xi |_v\cdot \mu(v)$,
with $v\in \mathbb{R}^n$ and the last dot for the action of
$\mathrm{Gl}_n$ on $\mathrm{Gl}_n/G$.

Therefore, a scalar differential invariant of $G$-structures
can be seen as a  differentiable function
$f\colon
	F^{r+1}M\bigl(J^r_{_0}(\mathbb{R}^n,\mathrm{Gl}_n/G)\bigr)\to \mathbb{R}$
which verifies
$f\circ \Bar{\varphi }^{r+1}=f$,
$\forall \varphi \in \mathcal{D}M$.
Now, as one application to $G$-structures of the theorem of
independence from $M$ of the invariants, we obtain:
\begin{teor}[Scalar differential $G$-invariants of $r$-th order]
	The set of scalar differential invariants
	of $r$-th order of $G$-structures on a
	manifold $M$
	is in natural bijective correspondence with the functions
	$h\colon \mathrm{G}^{r+1}_n\backslash 
		J^r_{_0}(\mathbb{R}^n,\mathrm{Gl}_n/G)\to \mathbb{R}$
	such that $h\circ \Pi $ is differentiable. 
	We say that $h$ is a \emph{
		scalar differential $G$-invariant of $r$-th order.}
\end{teor}

\subsection{Minimum number of scalar differential $G$-invariants}
It can be proved that the subspace of
$\mathrm{G}^{r+1}_n\backslash
	J^r_{_0}(\mathbb{R}^n,\mathrm{Gl}_n/G)$, 
denoted with ${G}^r$,
which is the union of the orbits of maximal dimension -- in 
$J^r_{_0}(\mathbb{R}^n,\mathrm{Gl}_n/G)$ --
is an open dense subset which -- I conjecture -- 
is a differentiable manifold
whose dimension, $m_r$, 
will be \emph{the number
	of functionally independent scalar differential $G$-invariants of $r$-th order}.

\begin{teor}[Minimum number of scalar differential $G$-invariants]
	Let $G$ be a closed subset of ${Gl}_n$ and $m=\dim G$. The number $m_r$
	of functionally independent scalar differential $G$-invariants of $r$-th order
	verifies
	\vspace{-2mm}
	\begin{equation*}
		m_r\geq (n^2-m)\binom{n+r}{n}-n\binom{n+r+1}{n}+n.
		\end{equation*}
	\vspace{-2mm}
\end{teor}

\begin{proof}
	Taking into account that
	$\dim J^r_{_0}(\mathbb{R}^n,\mathrm{Gl}_n/G)=(n^2-m)\binom{n+r}{n}$ and
	that $\dim \mathrm{G}^{r+1}_n=n\bigl(\binom{n+r+1}{n}-1\bigr)$,
	the result follows of
	$\dim (\mathrm{G}^{r+1}_n\backslash {G}^r)
	\geq \dim {G}^r -\dim \mathrm{G}^{r+1}_n$.
\end{proof}

\begin{table}[h]
	\centering
	\begin{tabular}{c|c|c|c|c|c}
		$_n\!\diagdown \!^r$  & 1 & 2 & 3 & 4 & 5 \\ \hline
		1 & 0 & 0 & 0 & 0 & 0 \\ \hline
		2 & - & {\bf 0}$^\dag$& 2 & 5 & 9 \\ \hline
		3 & - & 3 & 18 & 45 & 87 \\ \hline
		4 & - & 14 & 74 & 200 & 424 \\ \hline
		5 & - & 40 & 215 & 635 & 1475 \\ \hline
		6 & - & 90 & 510 & 1644 & 4164
		
	\end{tabular}
	\caption{Minimum number of $O_n$-invariants}
\end{table}
\begin{table}[h]
	\centering
	\begin{tabular}{c|c|c|c|c|c}
		$_n\!\diagdown \!^r$  & 1 & 2 & 3 & 4 & 5 \\ \hline
		1 & - & - & - & - & - \\ \hline
		2 & - & - & - & - & - \\ \hline
		3 & - & - & - & 10 & 31 \\ \hline
		4 & - & - & 39 & 130 & 298  \\ \hline
		5 & - & 19 & 159 & 509 & 1223 \\ \hline
		6 & - & 62 & 426 & 1434 & 3702
		
	\end{tabular}
	\caption{Minimum number of $CO_n$-invariants}
\end{table}
In the case of parallelizations of $M$ -- $G=\{I_n\}$ --
or in the case of fields of projective frames -- $G=\{kI_n:k\neq 0\}$ --
the minimum number of the theorem coincides 
with the exact number of invariants.

In the metric case the minimum  coincides
	with the exact number of $O_n$-invariants,
	except for $n=2$ and $r=2$ in which there is an invariant [$^\dag$ in Table 1].

In the conformal case, the exact number of $CO_n$-invariants is an open problem [see ($^\ast$) in the references].

\medskip

{\small \noindent
\textbf{Author address:}
\smallskip

\noindent (\textsc{Ignacio S\'anchez Rodr\'{\i}guez}) \\
\textsc{Departamento de Geometr\'{\i}a
y Topolog\'{\i}a, Facultad de Ciencias,
Universidad de Granada, Avda.\
Fuentenueva s/n, 18071-Granada, Spain}

\noindent \textit{E-mail:\/}
\texttt{ignacios@ugr.es}

}

\end{document}